\documentclass[journal,onecolumn]{IEEEtran}

\usepackage{amsmath}
\usepackage{amssymb}
\usepackage{amsthm}
\usepackage{caption2}
\usepackage{multirow}

\usepackage{enumitem}
\usepackage{bbm}
\usepackage{amsfonts}
\usepackage{mathrsfs}
\usepackage{graphicx}
\usepackage{epic}
\usepackage{epstopdf}
\usepackage{color}
\usepackage{cite}
\usepackage{mathtools}
\usepackage{enumerate}
\usepackage{lineno}
\usepackage{url}
\usepackage{float}
\usepackage{color}
\usepackage{bm}
\usepackage{tikz}

\usepackage[numbers]{natbib}
\newtheorem{theorem}{Theorem}[section]

\newtheorem{corollary}[theorem]{Corollary}

\newtheorem{definition}[theorem]{Definition}

\begin{document}
\title{Improved Lower Bounds for Strongly Separable Matrices and Related Combinatorial Structures}

\author{Bingchen Qian, Xin Wang and Gennian Ge
\thanks{The research of X. Wang was supported by the National Natural Science Foundation of China under Grant No. 11801392.  The research of G. Ge was supported by the National Natural Science Foundation of China under Grant No. 11971325, National Key Research and Development Program of China under Grant Nos. 2020YFA0712100 and 2018YFA0704703, and Beijing Scholars Program.}
\thanks{B. Qian is with the School of Mathematical Sciences, Capital Normal University, Beijing 100048, China. He is also  with the School of Mathematical Sciences, Zhejiang University, Hangzhou 310027, Zhejiang, China (email: qianbingchen@zju.edu.cn).}
\thanks{X. Wang is with the Department of Mathematics, Soochow University, Suzhou 215005, Jiangsu, China (email: xinw@suda.edu.cn).}
\thanks{G. Ge is with the School of Mathematical Sciences, Capital Normal University,
Beijing 100048, China (e-mail: gnge@zju.edu.cn).}
}

\maketitle



\begin{abstract}

In nonadaptive group testing, the main research objective is to design an efficient algorithm to identify a set of up to $t$ positive elements among $n$ samples with as few tests as possible. Disjunct matrices and separable matrices are two classical combinatorial structures while one provides a more efficient decoding algorithm and the other needs fewer tests, i.e., larger rate. Recently, a notion of strongly separable matrix has been introduced, which has the same identifying ability as a disjunct matrix, but has larger rate. In this paper, we use a modified probabilistic method to improve the lower bounds for the rate of strongly separable matrices. Using this method, we also improve the lower bounds for some well-known combinatorial structures, including locally thin set families and cancellative set families.



\smallskip


\end{abstract}

\section{Introduction}

Group testing was introduced by Dorfman \cite{10.2307/2235930} in 1943 and has been well-known for its various applications in blood testing, chemical leak testing, electric shorting detection, codes, multi-access channel communication and so on.  There are $n$ samples each can be either positive (used to be called defective) or negative (used to be called good). The problem is to identify all positive samples. Instead of testing one by one, group testing was proposed to pool all the samples into groups and perform a test to each group. If the testing outcome of a group is positive, it means that at least one positive sample is contained in this group. If the testing outcome of a group is negative, then this group contains no positive samples. The goal of group testing is to minimize the number of such tests in identifying all the positive samples.

There are two general types of group testing algorithms, adaptive (or sequential) and nonadaptive. An adaptive algorithm conducts the tests one by one and allows a later test to use the outcomes of all previous tests. A nonadaptive algorithm specifies all tests simultaneously, thus forbidding using the outcome information of one test to design another test. Adaptive algorithms require fewer number of tests in general, since extra information allows for more efficient test designs. Nonadaptive algorithms permit to conduct all tests simultaneously, thus saving the time for testing if not the number of tests. On the other hand, group testing can be roughly divided into probabilistic and combinatorial models. In probabilistic group testing, the positive samples are assumed to follow some probability distribution, while in combinatorial group testing, the number of positive samples is usually assumed to be no more than a fixed positive integer. In this paper, we consider only nonadaptive combinatorial group testing.

It was shown in \cite{MR1742957} that a $t$-disjunct matrix ($t$-DM) and a $\overline{t}$-separable matrix ($\overline{t}$-SM) could be utilized in nonadaptive group testing to identify any set of positives with size no more than $t$, but both have their advantages and disadvantages. Roughly speaking, a $t$-DM provides a more efficient identifying algorithm than a $\overline{t}$-SM, but a $\overline{t}$-SM implies a higher rate than a $t$-DM. To combine the advantages of these two structures, in \cite{MR4192047}, Fan et al. introduced a new notion of strongly $t$-separable matrix ($t$-SSM) for nonadaptive group testing which has weaker requirements than a $t$-DM but has the same identifying ability as a $t$-DM.

In this paper, we provide some new probabilistic approaches to improve the lower bounds for the rate of $t$-SSMs and some related well-known combinatorial structures, including locally thin set families and cancellative set families. We first introduce their definitions and recall some backgrounds on them.

\subsection{Strongly separable matrices}

A nonadaptive combinatorial group testing scheme can be represented as a $0$-$1$ (or binary) matrix $B=(b_{ij})$ whose columns are labeled by samples and rows by tests. Thus $b_{ij}$ specifies that test $i$ contains sample $j.$ It is more convenient to view a $0$-$1$ column $c_j$ as the incidence vector of subset $\{i\mid b_{ij}=1\}.$ Then we can talk about the union of a set of columns, which is nothing but the boolean sum of the corresponding $0$-$1$ columns. We first give the definitions of disjunct matrices and separable matrices which can be found in \cite{MR1742957}.
\begin{definition}
  Let $n,M,t\ge2$ be integers and $B$ be a binary matrix of size $n\times M$.
  \begin{itemize}
    \item $B$ is called a $t$-disjunct matrix, or briefly $t$-DM, if the Boolean sum of any $t$ column vectors of $B$ does not cover any other one.
    \item $B$ is called a $\overline{t}$-separable matrix, or briefly $\overline{t}$-SM, if the Boolean sums of  $\le t$ column vectors of $B$ are all distinct.
    \item $B$ is called a $t$-separable matrix, or briefly $t$-SM, if the Boolean sums of $t$ column vectors of $B$ are all distinct.
  \end{itemize}
\end{definition}

Let $M,n,t$ be integers with $M\ge n\ge 2$, and $B$ be a binary matrix of size $n\times M.$ Denote $[n] = \{1, \ldots, n\}$ and $[M] =\{1, \ldots, M\}.$ Let $\mathcal{F}= \{c_1,\ldots , c_M\} \subseteq \{0, 1\}^n$ be the set of column vectors of $B$ where $c_j = (c_j(1), \ldots , c_j(n)) \in \{0, 1\}^n$ for any $j\in [M].$ We say a vector $c_j$ covers a vector $c_k$ if for any $i\in [n], c_k(i) = 1$ implies $c_j(i) = 1.$

\begin{definition}\label{def of ssm}
  An $n\times M$ binary matrix $B$ is called a strongly $t$-separable matrix, or briefly $t$-SSM, if for any $\mathcal{F}_0\subseteq\mathcal{F}$ with $|\mathcal{F}_0|=t,$ we have
  \begin{equation*}
    \bigcap_{\mathcal{F}'\in\mathcal{U}(\mathcal{F}_0)}\mathcal{F}'=\mathcal{F}_0,
  \end{equation*}
  where
  \begin{equation*}
    \mathcal{U}(\mathcal{F}_0)=\left\{\mathcal{F}'\subseteq\mathcal{F}:\bigvee_{c\in\mathcal{F}_0}c=\bigvee_{c\in\mathcal{F}'}c \right\}.
  \end{equation*}
\end{definition}

%
%
%
%
%
The following relations are well known, see \cite{MR1742957, MR4192047}.
\begin{align}\label{relationship}
  \overline{t+1}\text{-separable}\Rightarrow t\text{-disjunct}\Rightarrow& \text{strongly } t\text{-separable} \Rightarrow\overline{t}\text{-separable}.
\end{align}

In \cite{MR2282446}, Du and Hwang proved that a binary matrix is $\overline{t+1}$-separable if and only if it is $(t+1)$-separable and $t$-disjunct. Here the connection between $\overline{t}$-separable and strongly $t$-separable is introduced.

\begin{theorem}
  Let $t$ be a positive integer and $B$ be a $(t+1)$-SM. Then $B$ is $\overline{t}$-separable if and only if $B$ is strongly $t$-separable.
\end{theorem}
\begin{proof}
  The sufficiency follows from (\ref{relationship}). So we only need to prove the necessity. Let $B$ be a $(t+1)$-SM and $\mathcal{F}$ be the set of column vectors of $B.$
  Let $\mathcal{F}_1=\{c_1,\ldots,c_t\}\subseteq\mathcal{F}$ with size $t,$ assume that there exists $\mathcal{F}_2$ with size $t+1$ such that $\bigvee_{c\in \mathcal{F}_1}c=\bigvee_{c\in\mathcal{F}_2}c$. If $\mathcal{F}_1\nsubseteq\mathcal{F}_2,$ by the property of $(t+1)$-separable, we have that $\bigvee_{i=1}^{t}c_i\neq\left(\bigvee_{i=1}^{t}c_i\right)\vee c',$ for all $c'\notin\mathcal{F}_1,$  so $\mathcal{U}(\mathcal{F}_1)=\mathcal{F}_1$ by the property that $B$ is also $\bar{t}$-separable. Then we obtain $\bigcap_{\mathcal{F}'\in\mathcal{U}(\mathcal{F}_1)}\mathcal{F}'=\mathcal{F}_1$ implying that $\mathcal{F}_1$ satisfies the condition in 
  Definition \ref{def of ssm}. If $\mathcal{F}_1\subseteq\mathcal{F}_2,$ let $\mathcal{F}_2=\{c_1,\ldots,c_t,c_{t+1}\},$ also by the property of $(t+1)$-separable, we have that  $\bigvee_{i=1}^{t}c_i\neq\left(\bigvee_{i=1}^{t}c_i\right)\vee c',$ for all $c'\neq c_i$ for $i\in\{1,2,\ldots,t+1\}.$ Thus, we have $\mathcal{U}(\mathcal{F}_1)=\{\mathcal{F}_1, \mathcal{F}_2\}.$ Similarly, we obtain $\bigcap_{\mathcal{F}'\in\mathcal{U}(\mathcal{F}_1)}\mathcal{F}'=\mathcal{F}_1$ implying that $\mathcal{F}_1$  satisfies the condition in 
  Definition \ref{def of ssm}. If there is no $\mathcal{F}_2$ with size $t+1$ such that $\bigvee_{c\in \mathcal{F}_1}c=\bigvee_{c\in\mathcal{F}_2}c$, then $\bigvee_{i=1}^{t}c_i\neq\left(\bigvee_{i=1}^{t}c_i\right)\vee c',$ for all $c'\notin\mathcal{F}_1,$ meaning that $\bigvee_{i=1}^{t}c_i$ can not cover any other element not in $\mathcal{F}_1.$ Since $B$ is $\bar{t}$-separable, we obtain $\bigcap_{\mathcal{F}'\in\mathcal{U}(\mathcal{F}_1)}\mathcal{F}'=\mathcal{F}_1$ implying that $\mathcal{F}_1$ satisfies the condition in 
   Definition \ref{def of ssm}. Since $\mathcal{F}_1$ can be chosen arbitrarily, $B$ is a $t$-SSM according to Definition \ref{def of ssm}.
\end{proof}

Let $SSM(t,n), DM(t,n)$ and $S(\overline{t},n)$ denote the maximum possible number of columns of a $t$-SSM, a $t$-DM and a $\overline{t}$-SM with $n$ rows, respectively. Denote their largest rates as
\begin{align*}
  R(t) &= \overline{\lim\limits_{n\to \infty}}\frac{\log_2SSM(t,n)}{n}, \\
  R_D(t) &=\overline{\lim\limits_{n\to \infty}}\frac{\log_2DM(t,n)}{n}, \\
  R_S(\overline{t}) & =\overline{\lim\limits_{n\to \infty}}\frac{\log_2S(\overline{t},n)}{n}.
\end{align*}

The following theorem combines the best known upper and lower bounds about $R_D(t)$ and $R_S(\overline{t}),$ for more details, please refer to the references \cite{MR1633846,MR1742957,MR677569,MR711896,MR1017407}.
\begin{theorem}[\cite{MR1633846,MR1742957,MR677569,MR711896,MR1017407}]
  Let $t\ge 2$ be an integer. If $t\rightarrow\infty,$ then we have
  \begin{align*}
    \frac{1}{t^2\log_2e}(1+o(1))\le R_D(t)
                                \le R_S(\overline{t})
                                \le R_D(t-1)\le\frac{2\log_2(t-1)}{(t-1)^2}(1+o(1)),
  \end{align*}
  where $e$ is the base of the natural logarithm. Moreover,
  \begin{align*}
    0.1814 &\le  R_D(2)\le 0.3219,\\
    0.3135&\le R_S(2) \le 0.4998.
  \end{align*}
\end{theorem}

By the relationships among SSM, DM and SM, the following corollary is straightforward.
\begin{corollary}[\cite{MR4192047}]\label{rate of R(t) frome relations}
  Let $t\ge2$ be an integer. Then we have
  $$\frac{1}{t^2\log_2e}(1+o(1))\le R(t)\le\frac{2\log_2(t-1)}{(t-1)^2}(1+o(1))$$
  for $t\rightarrow\infty,$ and
  $$0.1814\le R(2)\le 0.4998.$$
\end{corollary}

In \cite{MR4192047}, Fan et al. improved the lower bound for $R(2).$
\begin{theorem}[\cite{MR4192047}]
  $R(2)\ge 0.2213$.
\end{theorem}

\subsection{Locally thin set families}
In this subsection, we introduce some combinatorial structures whose lower bounds can be improved by the method we use for $t$-SSMs.  Let $\mathcal{F}$ be a family of subsets of a ground set of $n$ elements. We can suppose w.l.o.g. that our ground set is $[n]= \{1,2,\ldots,n\}.$ Following \cite{MR1816097, MR1810146, MR1845141}, we say that the family is $k$-locally thin if, for any $k$ of its distinct members, at least one point $i\in [n]$ is contained in exactly one of them. 

There are two generalizations for $k$-locally thin family.
\begin{itemize}
    \item (see \cite{MR1860438}) A family is called $k$-locally $2$-thin if, for any $k$ distinct members, there exists at least one point contained in $1$ or $2$ members.
  \item (see \cite{MR2900055}) A family is called locally $(k,b)$-thin if, for any $k$ distinct members, there exist at least $b$ points such that each of them is contained in exactly $1$ member.
\end{itemize}
Notice that if a family is $k$-locally thin then it is also $(k+1)$-locally $2$-thin.

Since the first generalization is weaker than the original one and the second is stronger, we
let $WM(n,k,2)$ and $SM(n, k, b)$ denote the maximum cardinality of a $k$-locally $2$-thin family and a locally $(k,b)$-thin family of subsets of $[n]$, respectively. Denote their largest rates as
\begin{align*}
  wt(k,2) &= \overline{\lim\limits_{n\to \infty}}\frac{\log_2WM(n,k,2)}{n}, \\
  st(k,b) &=\overline{\lim\limits_{n\to \infty}}\frac{\log_2SM(n,k,b)}{n}.
\end{align*}

This problem was investigated by Alon, Fachini, K\"orner and Monti \cite{MR1816097, MR1810146, MR1845141}. They proved that $\frac{1}{3}(6-\log_{2}37)\le st(4,1)<0.4561\ldots$ and $st(k,1)<2/k$ for all even $k,$ and
$$ \Omega\left(\frac{1}{k}\right)\le st(k,1)\le O\left(\frac{\log_{2}k}{k}\right)<0.793$$
for all $k.$ This is a notoriously hard problem. In particular, we do not even know whether $st(3,1)<1.$

Later, in \cite{MR1860438}, the authors derived new lower and upper bounds for $st(5,1)$ and $wt(6,2).$
\begin{theorem}\rm (\cite{MR1860438})
  $$0.1900<\frac{\log_{2}\frac{625}{369}}{4}\le st(5,1)\le wt(6,2)<0.596.$$
\end{theorem}

\subsection{$t$-cancellative set families}
In this subsection, we introduce another related well-known structure, i.e., cancellative set families. A family of sets $\mathcal{F}$ (and the corresponding family of $0$-$1$ vectors) is called $t$-cancellative, if for all distinct $t+2$ members $A_1,\ldots,A_t$ and $B,C\in\mathcal{F}$,
$$A_1\cup\cdots\cup A_t\cup B\neq A_1\cup\cdots\cup A_t\cup C.$$
For the case $t=1,$ it is just called cancellative for simplicity.

Actually, as pointed out in \cite{MR2900055}, we can see that requiring the family $\mathcal{F}$ to be $t$-cancellative is equivalent to asking its representation set of $0$-$1$ vectors ($x:=x(A)=(x_1,x_2,\ldots,x_n),$ with $x_i=1$ if $i\in A$ and $x_i=0$ otherwise for $A\in\mathcal{F}$), to satisfy the following: for every $(t+2)$-tuple $(x^{(1)},x^{(2)},\ldots,x^{(t+2)})$ of distinct vectors in the set (considered in an arbitrary but fixed order), there exist at least $(t+1)$ different values of $k\in[n],$ such that the corresponding ordered $(t+2)$-tuples $(x^{(1)}_k,x^{(2)}_k,\ldots,x^{(t+2)}_k)$ are all different, while for each of them we have the sum $x^{(1)}_k+x^{(2)}_k+\cdots+x^{(t+2)}_k=1.$ That is, we can consider $t$-cancellative to be much stronger than $(t+2,t+1)$-locally thin.

Let $c(t,n)$ be the size of the largest cancellative family on $n$ elements. The problem was proposed by Erd\H{o}s and Katona \cite{MR0351890} for the case $t=1,$ and conjectured that $c(1,n)=\Theta(3^{n/3}).$ It was disproved by an elegant construction by Shearer \cite{MR1383741} showing that $c(1,3k)\ge k3^{k-2},$ leading to $c(1,n)>1.46^n$ for $n>n_0.$

The asymptotic of $c(1,n)$ was given in \cite{841182} (construction) and \cite{MR782061} (upper bound), showing that there exists a $\gamma>0$ such that
$$\frac{\gamma}{\sqrt{n}}1.5^n<c(1,n)<1.5^n.$$
Since a product of two cancellative families is again cancellative, we have $c(1,n+m)\ge c(1,n)c(1,m).$  Thus $\lim\limits_{n\to \infty} c(1,n)^{1/n}$ exists. This is not known for $t\geq 2$, so K\"{o}rner and Sinaimeri \cite{MR2346813} introduced
$$r(t):=\overline{\lim\limits_{n\to \infty}}\frac{\log_2 c(t,n)}{n}$$
for the case $t=2$ and proved $0.1139<r(2)\le 0.42.$  Later, F\"uredi \cite{MR2900055} improved this upper bound to $r(2)\le\log_25-2=0.3219\ldots.$

The rest of this paper is organized as follows: in Sections II, we will derive new lower bounds for the rates of $2$-separable matrices and give a general method to derive new lower bounds for $t$-separable matrices for $t\ge3.$ In Section III, we will discuss the lower bounds for some related combinatorial structures, i.e., locally thin set families and cancellative set families. We conclude in Section IV.

\section{Strongly $t$-separable matrices}

In this section, we will mainly concentrate on the largest rate of a $t$-SSM. We first 
derive an improved lower bound for a $2$-SSM. Next, we give  a general method to derive lower bounds for $t$-SSMs for $t\ge 3$ and exactly compute it for the case $t=3.$


\subsection{An improved lower bound for $R(2)$}

Inspired by the construction of Shearer in \cite{MR1383741}, in this part we provide an improved lower bound for $R(2)$ by a modified probabilistic method.

\begin{theorem}\label{thm of lower bound of R(2)}
  $R(2)\ge 0.2237$.
\end{theorem}

\begin{proof}
Without loss of generality, we assume that $45\mid n.$ Indeed, if not, we may replace $n$ by $45\lfloor\frac{n}{45}\rfloor$, which will not affect the rate since $n$ tends to infinity.
Partition $[n]$ into $n/45$ blocks each of size $45,$ and further partition each block into $15$ triplets. Within each triplet, assign labels $0,1,2$ to the elements.
For each subset in our family, for each block, select one of the triplets and take all three of its elements; select one element from each other triplet in the block, in such a way that a parity condition holds: the sum of the $14$ labels is divisible by $3.$ The number of choices for each block is then $15\times 3^{13},$ and the total number of subsets is
\begin{equation*}
  M=(15\times 3^{13})^{n/45}.
\end{equation*}
We denote this family by $\mathcal{H}.$ We claim that for arbitrary three distinct members $A, B,C\in\mathcal{H},$ we have that $A\cup B\neq A$ or $C$ and $A\cup B\neq A\cup C.$ Indeed, the first inequality is obvious. For the second inequality, assume that  $A\cup B= A\cup C,$ then $(B\triangle C)\subseteq A.$ If $B$ and $C$ select different triplets to take all three elements, then these two triplets in $(B\triangle C)$ contain two elements each. If $B$ and $C$ select the same triplets to take all three elements, then by the parity condition, there are at least two triplets in $(B\triangle C)$ containing two elements each. By the structure of $A,$ both cases contradict to that $(B\triangle C)\subseteq A.$

We choose each member of $\mathcal{H}$ independently at random with probability $p^{n/45}$ which will be determined later. We now count the number of bad structures that we may have selected.

There are two bad structures that we should forbid. The first case is $A\cup B=A\cup C_1\cup C_2\cup\ldots\cup C_s$ for some $s\ge 2$ while the second is $A\cup B=D_1\cup D_2\cup\ldots\cup D_s$ for some $s\ge 2,$ where $A,B,C_i,D_j\in\mathcal{H}.$  Actually, whenever we have the second case, we may add $A$ or $B$ in the right hand side to obtain the first case. Therefore, we only need to forbid the first case.

Now we count the number of the structures $A\cup B=A\cup C_1\cup C_2\cup\ldots\cup C_s,$ for some $s\ge 2$.
Consider a particular block of $45$ elements. If there is only one triplet on which $A\cup B$ has all three elements, then the special triplet chosen by $A$ is the same one chosen by $C_i.$ There are $15$ choices of location for this triplet. For each of $14$ other triplets, either the union has one element (in which case $B,C_i$ agree with the choice made by $A$),  or the union has two elements, in which case $B$ disagrees with $A,$ $C_i$ agrees with either $A$ or $B,$ but there must be at least one $C_i$ that agrees with $B$. The total number of choices on this triplet is at most
\begin{equation*}
  3+6(2^s-1)=3\times 2^{s+1}-3.
  \end{equation*}
Values on the last triplet are forced. The number of choices for one block, in this case, is then at most $15\times(3\times 2^{s+1}-3)^{13}.$

If $A\cup B$ contains two triplets with all elements, which means that $B$ makes a different choice of the triplet chosen by $A$, then the all-elements triplet chosen by $C_i$  must agree with either $A$ or $B$. The number of such choices is at most
\begin{equation*}
  15\times 14\times 2^s.
\end{equation*}
The other $13$ triplets again allow $3\times 2^{s+1}-3$ choices. So the number of choices for one block, in this case, is at most $210\times 2^s\times(3\times 2^{s+1}-3)^{13}.$

Since we should also forbid the structures $A\cup B=B\cup C_1\cup C_2\cup\ldots\cup C_s$ for some $s\ge 2,$ if the following inequality holds
\begin{align*}
  2\sum_{s=2}^{\infty}\left((15\times(3\times 2^{s+1}-3)^{13}+\right.
  210\left.\times 2^s\times(3\times 2^{s+1}-3)^{13})p^{s+2}\right)^{n/45} \ll (15\times 3^{13}p)^{n/45},
\end{align*}
then we can remove one element from each bad structure and the family is strongly $2$-separable with size at least, say $(15\times 3^{13}p)^{n/45}/2.$ 
Dividing both sides by $(15\times 3^{13}p)^{n/45}$ and the above inequality becomes
\begin{equation*}
  2\sum_{s=2}^{\infty}\left(\frac{(210\times 2^s+15)(6\times2^s-3)^{13}}{15\times 3^{13}}p^{s+1}\right)^{n/45}\ll 1,
\end{equation*}
which is
\begin{equation*}
  2\sum_{s=2}^{\infty}\left((14\times 2^s+1)(2^{s+1}-1)^{13}p^{s+1}\right)^{n/45}\ll 1.
\end{equation*}

Now we estimate the left hand side of the above inequality. Let $f(s)=(14\times 2^s+1)(2^{s+1}-1)^{13}p^{s+1}.$ So if we have that $(2^{14}p)^{n/45}<1,$ then

\begin{align*}
  \sum_{s=2}^{\infty}f(s)^{n/45} =&\sum_{s=2}^{5}f(s)^{n/45}+\sum_{s=6}^{\infty}f(s)^{n/45} \\
   \le& \sum_{s=2}^{5}f(s)^{n/45}+\sum_{s=6}^{\infty}\left(15\times 2^s\times2^{13(s+1)}p^{s+1}\right)^{n/45} \\
   =&\sum_{s=2}^{5}f(s)^{n/45}+(15\times 2^{13}p)^{n/45}\sum_{s=6}^{\infty}\left((2^{14}p)^{n/45}\right)^{s} \\
   =& \sum_{s=2}^{5}f(s)^{n/45}+(15\times 2^{13}p)^{n/45}\frac{(2^{14}p)^{6n/45}}{1-(2^{14}p)^{n/45}}.\\
\end{align*}
Taking $p=4.487\times10^{-5},$ we have that $f(s)<1$ for $s\in\{2,3,4,5\},$ $\left((15\times2^{13}p)\times(2^{14}p)^6\right)^{n/45}<1$ and $\left(2^{14}p\right)^{n/45}<1.$ Thus, when $n$ is sufficiently large we get $\sum_{s=2}^{\infty}f(s)^{n/45}\ll1/2$ as desired.

Therefore, we have shown the existence of a $2$-SSM of length $n$ and size $\left(4.487\times10^{-5}\times 15\times 3^{13} \right)^{n/45}/2.$ Thus,
\begin{equation*}
  R(2)\ge\lim\limits_{n\to \infty}  \frac{\log_2\left( \left(4.487\times10^{-5}\times 15\times 3^{13} \right)^{n/45}/2\right)}{n}\ge 0.2237.
\end{equation*}
\end{proof}

\subsection{Improved lower bound for $R(t)$ for $t\ge3$}

Similar to the case $t=2,$ we now give a general method to derive lower bounds for $R(t)$ where $t\ge3,$ before which, the only known results follow from Corollary \ref{rate of R(t) frome relations}.

 First we define $S(a,b):=\sum_{i=1}^{b}(-1)^{b-i}{b\choose i}i^a$ to be the number of surjective functions from a set $A=\{1,\ldots,a\}$ into a set $B=\{1,\ldots,b\}.$

  We  assume that $b\mid n,$ where $b$ is a constant positive integer which only depends on $t$ and will be determined with flexility. Indeed, if $b \nmid n,$ we may replace $n$ by $b\lfloor\frac{n}{b}\rfloor $ and it will not affect the rate. Partition $[n]$ into $\frac{n}{b}$ blocks of $b$ elements each. Within each block, select one element. The total number of subsets obtained in this way is
\begin{equation*}
  M=b^{n/b}.
\end{equation*}
We denote this family by $\mathcal{H}.$

We choose each member of $\mathcal{H}$ independently at random with probability $p^\frac{n}{b}$ which will be determined later. We now count the number of bad structures that we may have selected.

For $0\le t'\le t-1,$ we should forbid
the bad structures: $A_1\cup\cdots\cup A_t=A_{i_1}\cup\cdots\cup A_{i_{t'}}\cup C_1\cup C_2\cup\cdots\cup C_s$ for some $s\ge 0.$ Since whenever we have $A_1\cup\cdots\cup A_t=A_{i_1}\cup\cdots\cup A_{i_{t'}}\cup C_1\cup C_2\cup\cdots\cup C_s,$ we will have $A_1\cup\cdots\cup A_t=A_{i_1}\cup\cdots\cup A_{i_{t'+1}}\cup C_1\cup C_2\cup\cdots\cup C_s$ for some $0\le t'\le t-2,$ we can only focus on the case $A_1\cup\cdots\cup A_t=A_{i_1}\cup\cdots\cup A_{i_{t-1}}\cup C_1\cup C_2\cup\cdots\cup C_s,$ where $\{i_1,i_2,\ldots,i_{t-1}\}\subseteq\{1,2,\ldots,t\}$.


Now we focus on the case $A_1\cup\cdots\cup A_t=A_1\cup\cdots\cup A_{t-1}\cup C_1\cup C_2\cup\cdots\cup C_s.$
Consider a particular block of $b$ elements, 
we count the number of choices each $A_i$ can have.
The total number is at most
\begin{equation*}
  g(s):=\sum_{j=1}^{t-1}{b\choose j}S(t,j)j^s+{b\choose t}S(t,t)(t^s-(t-1)^s),
\end{equation*}
where ${b\choose j}$ means the number of choices of $j$ positions where $A_1\cup\cdots\cup A_t$ agrees with, $S(t,j)$ means the number of ways how $t$ elements in $A_1\cup\cdots\cup A_t$ agree with the chosen $j$ positions, 
$j^s$ means the number of choices all $C_i$ can have. And the last term is the case $j=t,$ it is special because if $A_1\cup\cdots\cup A_t$ contains $t$ positions, then to make $A_1\cup\cdots\cup A_t=A_1\cup\cdots\cup A_{t-1}\cup C_1\cup C_2\cup\cdots\cup C_s$ true, the elements on this block of $C_1\cup C_2\cup\cdots\cup C_s$ can not agree only with the $(t-1)$ positions where $A_1\cup\cdots\cup A_{t-1}$ has.


Therefore, if we have that
\begin{equation} \label{inequlity}
  t\sum_{s=0}^{\infty}\left(g(s)p^{s+t}\right)^{n/b}\ll \left (bp\right)^{n/b},
\end{equation}
where $t$ means that there are $t$ choices to choose $A_{i_1},\ldots,A_{i_{t-1}}$(and also note that it is a constant and the deletion of the factor $t$ in the left hand side of (\ref{inequlity}) will not affect the inequality, so for simplicity, we ignore it in the computation for the case $t=3$ later),
then we can remove one element from each bad structure and the family is strongly $t$-separable with size at least, say $\left(bp\right)^{n/b}/2,$ which has rate at least
\begin{equation}\label{rate formular}
  \frac{\log_2(bp)}{b}.
\end{equation}

For the choice of parameters $b$ and $p,$ we only need the inequality (\ref{inequlity}) to be satisfied.

Here we calculate the case $t=3,$ in the literature, the best known lower bound for a $3$-SSM follows from that $R(3)\ge R_D(3)\ge 0.079.$  Now we give an improved lower bound.
\begin{theorem}\label{theorem R(3)}
  $R(3)\ge 0.0974.$
\end{theorem}
\begin{proof}
  We substitute $t=3$ in the process discussed above, we get that
  $$g(s):=b(b-1)(b-2)3^s-b(b-1)(b-5)2^s+b.$$
  Now we need to determine $b$ and $p$ such that the following is satisfied
  $$\sum_{s=0}^{\infty}\left(g(s)p^{s+3}\right)^{n/b}\ll \left(bp\right)^{n/b}, $$
  which follows from (\ref{inequlity}) and as we said before, we omit the coefficient $3$ in the left hand side.
  Dividing both sides by the right hand side, we get
  \begin{equation}\label{t=3}
  \sum_{s=0}^{\infty}\left(((b-1)(b-2)3^s-(b-1)(b-5)2^s+1)p^{s+2}\right)^{n/b}\ll 1.
  \end{equation}
  Assume that $b\ge 6$ and $3p<1,$ we have that
  \begin{align*}
    & \text{LHS of } (\ref{t=3})   \\
     \le& \left((3b-2)p^2\right)^{n/b}+ \sum_{s=1}^{\infty}\left((b-1)(b-2)3^s p^{s+2}\right)^{n/b} \\
     \le& \left((3b-2)p^2\right)^{n/b}+\left((b-1)(b-2)p^2\right)^{n/b}\sum_{s=1}^{\infty}\left((3p)^{n/b}\right)^{s} \\
     \le& \left((3b-2)p^2\right)^{n/b}+\left((b-1)(b-2)p^2\right)^{n/b}\frac{(3p)^{n/b}}{1-(3p)^{n/b}} \\
     \le& \left((3b-2)p^2\right)^{n/b}+\frac{\left(3(b-1)(b-2)p^3\right)^{n/b}}{1-(3p)^{n/b}}.
  \end{align*}
  Set $b=6$ and $p=0.24999,$ and let $n$ tend to infinity, it is easy to check that the following is true
  $$\left((3b-2)p^2\right)^{n/b}+\frac{\left(3(b-1)(b-2)p^3\right)^{n/b}}{1-(3p)^{n/b}} \ll 1.$$

  Therefore, by (\ref{rate formular}), we have that
  $$R(3)\ge 0.0974.$$

\end{proof}

\section{Applications for combinatorial structures}
In this section, we give improved lower bounds for some combinatorial structures, including $k$-locally thin set families and $t$-cancellative set families.

\subsection{$k$-locally thin families}
In this subsection, we give improved lower bounds for both $5$-locally thin and $6$-locally $2$-thin set families.
\begin{theorem}\label{theorem lower bound for 5-loaclly thin}
  $st(5,1)> 0.1965.$
\end{theorem}
\begin{proof}
  Let $\mathcal{H}$ be the family as defined in Section II.B. Also, we choose each member of $\mathcal{H}$ independently at random with probability $p^{\frac{n}{b}}$ which will be determined later.

  We now count the number of bad structures that we may have selected.
  For any $5$ distinct members $A_1,A_2,\ldots,A_5$ in $\mathcal{H},$ we say that it is bad if there is no point that belongs to exactly one member. Consider a particular block of $b$ elements, we count the number of choices each $A_i$ can have. It is at most
  $$ h(b) = b + b(b-1){5\choose 3},$$
  where the first term in the right hand side means the number of choices that all the $5$ elements agree with one position among $b$ positions while the second term means that these $5$ elements agree with two positions and $3$ of them agree with one and the other $2$ agree with another, for simplicity, we denote this by $\{3,2\}$-type. It is not hard to check that these are the only two bad structures. Indeed, if $5$ elements agree with $2$ positions, then the only other case is $\{4,1\}$-type, which is not bad; If $5$ elements agree with more than $2$ positions, then by pigeonhole, there exists at least one position that has only one element. Thus, if we have that
  $$ \left(h(b)p^5\right)^{n/b} \ll (bp)^{n/b},$$
  we can remove one set from each bad structure and the family is $5$-locally thin.

  Set $b=5$ and $p=0.39518,$ and let $n$ tend to infinity, then the above inequality is true, and we get a $5$-locally thin set family with rate at least
  $$ \frac{\log_2(bp)}{b}>0.1965,$$
  which completes the proof.
\end{proof}

\begin{theorem}
  $wt(6,2) > 0.2522.$
\end{theorem}
\begin{proof}
  The proof is similar to that of Theorem \ref{theorem lower bound for 5-loaclly thin}. Let $\mathcal{H},b,p$ be the parameters of the same meaning as in Theorem \ref{theorem lower bound for 5-loaclly thin}.

  For any $6$ distinct members $A_1,A_2,\ldots,A_6$ in $\mathcal{H},$  we say that it is bad if there is no point that belongs to $1$ or $2$ members. Consider a particular block of $b$ elements, we count the number of choices each $A_i$ can have. It is at most
  $$h(b) = b+{b\choose 2}{6\choose 3},$$
  where the first term in the right hand side means the number of choices such that all $6$ elements agree with one position while the second term means that the number of ways that $6$ elements agree with two positions, $3$ agree with one and the other $3$ agree with another, i.e., it is $\{3,3\}$-type. We claim that these are the only bad structures. Indeed, if $6$ elements agree with two positions, then the other conditions are $\{5,1\}$-type and $\{4,2\}$-type, which are both not bad; if $6$ elements agree with at least three positions, then by pigeonhole, there exists at least one position that has one or two elements. Thus, if we have that
  $$ \left(h(b)p^6\right)^{n/b}\ll (bp)^{n/b},$$
  we can remove one set from each bad structure and the family is $6$-locally $2$-thin.

  Set $b=4,$ $p= 0.50318,$ and let $n$ tend to infinity, then the above inequality is true, and we get a $6$-locally $2$-thin set family with rate at least
  $$ \frac{\log_2(bp)}{b}>0.2522,$$
  which completes the proof.
\end{proof}

\subsection{$t$-cancellative set families}
In this subsection, we give a modified probabilistic method to improve the lower bound for $r(2)$. For $t\ge 3,$ we could just modify the method we use in Section II.B by setting $s=1.$

\begin{theorem}
  $r(2)>0.1170.$
\end{theorem}
\begin{proof}
  The proof is similar to that of Theorem \ref{theorem R(3)}. The only difference is that here we should just focus on the case $s=1.$ That is, we need to forbid the case $A_1\cup A_2\cup A_3=A_1\cup A_2\cup C_1$ for any distinct $A_1,A_2,A_3,C_1\in \mathcal{H},$ where $\mathcal{H}$ is defined in Section II.B.  Use the same notation, it suffices to show that the following is true
  $$(g(1)p^4)^{n/b}\ll (bp)^{n/b},$$
  where $g(1)=b+6b(b-1)+b(b-1)(b-2).$
  Setting $b=5$ and $p=0.3001,$ we have that
  $$r(2)\ge \frac{\log_2(bp)}{b}>0.1170.$$
\end{proof}

\section{Conclusions}

In this paper, we consider the lower bounds for  the rates of several interesting structures, such as strongly $t$-separable matrices, locally thin set families and cancellative set families.  As a consequence, by a modified probabilistic construction, we improve the existing lower bounds.  However, there is still a gap between the known upper bounds and lower bounds for these structures. It would be of interest to narrow the gap by new methods.

\bibliographystyle{abbrv}
\bibliographystyle{IEEEtran}
\bibliography{REF}

\begin{thebibliography}{10}

\bibitem{MR1816097}
N.~Alon, E.~Fachini, and J.~K\"{o}rner.
\newblock Locally thin set families.
\newblock {\em Combin. Probab. Comput.}, 9(6):481--488, 2000.

\bibitem{MR1810146}
N.~Alon, J.~K\"{o}rner, and A.~Monti.
\newblock String quartets in binary.
\newblock {\em Combin. Probab. Comput.}, 9(5):381--390, 2000.

\bibitem{MR1633846}
D.~Coppersmith and J.~B. Shearer.
\newblock New bounds for union-free families of sets.
\newblock {\em Electron. J. Combin.}, 5:Research Paper 39, 16, 1998.

\bibitem{10.2307/2235930}
R.~Dorfman.
\newblock The detection of defective members of large populations.
\newblock {\em Ann. math. stat.}, 14(4):436--440, 1943.

\bibitem{MR1742957}
D.-Z. Du and F.~K. Hwang.
\newblock {\em Combinatorial group testing and its applications}, volume~12 of
  {\em Series on Applied Mathematics}.
\newblock World Scientific Publishing Co., Inc., River Edge, NJ, second
  edition, 2000.

\bibitem{MR2282446}
D.-Z. Du and F.~K. Hwang.
\newblock {\em Pooling designs and nonadaptive group testing}, volume~18 of
  {\em Series on Applied Mathematics}.
\newblock World Scientific Publishing Co. Pte. Ltd., Hackensack, NJ, 2006.
\newblock Important tools for DNA sequencing.

\bibitem{MR711896}
A.~G. D'yachkov and V.~V. Rykov.
\newblock Bounds on the length of disjunctive codes.
\newblock {\em Problemy Peredachi Informatsii}, 18(3):7--13, 1982.

\bibitem{MR1017407}
A.~G. D'yachkov, V.~V. Rykov, and A.~M. Rashad.
\newblock Superimposed distance codes.
\newblock {\em Problems Control Inform. Theory/Problemy Upravlen. Teor.
  Inform.}, 18(4):237--250, 1989.

\bibitem{MR677569}
P.~Erd\H{o}s, P.~Frankl, and Z.~F\"{u}redi.
\newblock Families of finite sets in which no set is covered by the union of
  two others.
\newblock {\em J. Combin. Theory Ser. A}, 33(2):158--166, 1982.

\bibitem{MR1845141}
E.~Fachini, J.~K\"{o}rner, and A.~Monti.
\newblock A better bound for locally thin set families.
\newblock {\em J. Combin. Theory Ser. A}, 95(2):209--218, 2001.

\bibitem{MR1860438}
E.~Fachini, J.~K\"{o}rner, and A.~Monti.
\newblock Self-similarity bounds for locally thin set families.
\newblock {\em Combin. Probab. Comput.}, 10(4):309--315, 2001.

\bibitem{MR4192047}
J.~Fan, H.-L. Fu, Y.~Gu, Y.~Miao, and M.~Shigeno.
\newblock Strongly separable matrices for nonadaptive combinatorial group
  testing.
\newblock {\em Discrete Appl. Math.}, 291:180--187, 2021.

\bibitem{MR782061}
P.~Frankl and Z.~F\"{u}redi.
\newblock Union-free hypergraphs and probability theory.
\newblock {\em European J. Combin.}, 5(4):395, 1984.

\bibitem{MR2900055}
Z.~F\"{u}redi.
\newblock 2-cancellative hypergraphs and codes.
\newblock {\em Combin. Probab. Comput.}, 21(1-2):159--177, 2012.

\bibitem{MR0351890}
G.~O.~H. Katona.
\newblock Extremal problems for hypergraphs.
\newblock In {\em Combinatorics ({P}roc. {NATO} {A}dvanced {S}tudy {I}nst.,
  {B}reukelen, 1974), {P}art 2: {G}raph theory; foundations, partitions and
  combinatorial geometry}, pages 13--42. Math. Centre Tracts, No. 56, 1974.

\bibitem{MR2346813}
J.~K\"{o}rner and B.~Sinaimeri.
\newblock On cancellative set families.
\newblock {\em Combin. Probab. Comput.}, 16(5):767--773, 2007.

\bibitem{MR1383741}
J.~B. Shearer.
\newblock A new construction for cancellative families of sets.
\newblock {\em Electron. J. Combin.}, 3(1):Research Paper 15, approx. 3, 1996.

\bibitem{841182}
L.~Tolhuizen.
\newblock New rate pairs in the zero-error capacity region of the binary
  multiplying channel without feedback.
\newblock {\em IEEE Trans. Inform. Theory}, 46(3):1043--1046, 2000.

\end{thebibliography}
\end{document}